\documentclass{amsart}

\usepackage{amsthm}

\numberwithin{equation}{section}
\newtheorem{thm}[equation]{Theorem}

\newtheorem{lem}[equation]{Lemma}
\newtheorem{prp}[equation]{Proposition}

\theoremstyle{remark}

\newtheorem{dfn}[equation]{Definition}
\newtheorem{rem}[equation]{Remark}

\newcommand{\EXP}{\ensuremath{\mathbb{E}}}

\newcommand{\var}{\operatorname{Var}}
\newcommand{\esp}{\mathbb{E}}
\newcommand{\pt}{\enspace .}
\newcommand{\ent}{\operatorname{Ent}}
\newcommand{\dis}{\sim}
\newcommand{\proba}{\mathbb{P}}

\begin{document}
\title{Concentration inequalities for  order statistics}
\author{St\'ephane Boucheron \and Maud Thomas}
\address{LPMA CNRS UMR 7599 \& Universit\'e Paris-Diderot} 
\date{\today}

\begin{abstract}
This note describes non-asymptotic variance and tail bounds for order statistics of samples
of independent identically distributed random variables. Those bounds are checked to be 
asymptotically  tight when the sampling distribution belongs to a maximum domain of 
attraction. If the sampling distribution has non-decreasing hazard rate (this includes the 
Gaussian distribution), we derive an exponential Efron-Stein inequality for order statistics: an inequality connecting the logarithmic moment generating function of centered order statistics 
with exponential moments of Efron-Stein (jackknife) estimates  of variance.  
We use this general connection to derive variance and tail bounds for order statistics
of Gaussian sample. Those bounds are not within the scope of the Tsirelson-Ibragimov-Sudakov 
 Gaussian concentration inequality. 
Proofs are elementary and combine R\'enyi's representation
of order statistics and the so-called entropy approach to concentration inequalities 
popularized by M. Ledoux.
\end{abstract}
\maketitle

\section{Introduction}
\label{sec:introduction}

The purpose of this note is to develop non-asymptotic variance and tail bounds for  order statistics.
In the sequel,  $X_1,\ldots,X_n$ are independent random variables, distributed according to some probability distribution $F$, and  
$X_{(1)}\geq X_{(2)} \geq \ldots \geq X_{(n)}$  denote  
 the corresponding order statistics (the non-increasing rearrangement of $X_1,\ldots,X_n$). 
The cornerstone of Extreme Value Theory (\textsf{EVT}), the Fisher-Tippett-Gnedenko Theorem 
 describes the asymptotic behavior of $X_{(1)}$ after centering and normalization \cite{dHaFei06}. 
The median $X_{(\lfloor n/2\rfloor)}$ is a widely used location estimator. Its asymptotic properties are well documented 
(See for example \cite{vandervaart:1998} for a review). 
Much less seems to be available  if the sample size $n$ is fixed. The distribution function of $X_{(1)}$
 is obviously explicitly known   ($F^n$ !), but simple and  useful variance or tail bounds do not seem to be publicized. 

Our main tools will be the Efron-Stein inequalities  that assert that the jackknife estimate(s)  of the 
variance of functions of independent random variables are on average upper bounds, and extensions of those inequalities
that allow to derive exponential tail bounds (See Theorem \ref{thm:ess:1}).

We refer to \cite{Mil74,ShaoWu89} and references therein for an account of the interplay 
between jackknife estimates, order statistics, extreme value theory and statistical inference.

The search for non-asymptotic variance and tail bounds for extreme order statistics is not only motivated by 
the possible applications  of \textsf{EVT} to quantitative risk management, but also by our desire 
to understand some aspects of the concentration of measure phenomenon \cite{ledoux:2001,massart:2003}. 
Concentration
of measure  theory  tells us that a function of many independent random variables that does not depend too much 
on any of them  is almost constant. The best known results in that field are the Poincar\'e and Gross logarithmic Sobolev inequalities
and the Tsirelson-Ibragimov-Sudakov tail bounds for functions of random Gaussian vectors. If $X_1,\ldots,X_n$
are independent standard Gaussian random variables, and $f\colon\mathbb{R}^n\rightarrow\mathbb{R}$ is $L$-Lipschitz,
then $Z=f(X_1,\ldots,X_n)$ satisfies $\operatorname{Var}(Z)\leq L^2$, $\log \mathbb{E}[\exp(\lambda(Z-\mathbb{E}Z))]\leq \tfrac{\lambda^2L^2}{2}$ 
and $\mathbb{P}\{ Z-\mathbb{E}Z\geq t\}\leq \exp(-\tfrac{t^2}{2L^2})\, .$ If we apply those bounds to $X_{(1)}$ (resp. to $X_{(\lfloor n/2\rfloor )}$) 
that is the maximum (resp. the median) of 
$X_1,\ldots,X_n$, the  Lipschitz constant  is (almost surely) 
 $L=1$, so Poincar\'e inequality allows to establish $\operatorname{Var}(X_{(1)})\leq 1$  (resp. $\operatorname{Var}(X_{(\lfloor n/2\rfloor )})\leq 1$ ).
This easy upper bound is far from being satisfactory, 
it is well-known in \textsf{EVT}  that $\operatorname{Var}(X_{(1)})= O\big(\tfrac{1}{\log n }\big)$ and 
 $\operatorname{Var}(X_{(\lfloor n/2\rfloor )})= O\big(\tfrac{1}{n}\big)$.
Naive use of off-the-shelf concentration bounds does not work when handling maxima or  
order statistics at large. 
This situation is not uncommon. The analysis
of the largest eigenvalue of  random matrices from the Gaussian Unitary Ensemble (\textsf{GUE}) \cite{Led03} 
provides  a setting where the derivation of sharp concentration inequalities require 
ingenuity and combining concentration/hypercontractivity with special representations. 

Our purpose is to show that the tools and methods used to investigate the  concentration of measure phenomenon are relevant 
to the analysis of order statistics. When properly combined with R\'enyi's representation for order statistics (see Theorem \ref{thm:renyi}) 
the so-called entropy method  developed and popularized by Ledoux \cite{ledoux:2001}  allows to recover sharp variance and tail bounds. 
Proofs are elementary and parallel the approach followed in \cite{Led03} in a much more sophisticated setting: whereas Ledoux 
builds on the determinantal structure of the joint density of the eigenvalues of random  matrices from the \textsf{GUE} 
to upper bound tail bounds by sums of Gaussian integrals that can be handled by concentration/hypercontractivity
  arguments, in the sequel,
we build on R\'enyi's representation of order statistics: $X_{(1)},\ldots,X_{(n)}$ can be represented as the image 
of the order statistics of a sample of the exponential distribution by a monotone function. The order statistics 
of an exponential sample turn out to be represented as partial sums of independent random variables.

In Section~\ref{sec:gener-purp-ineq}, using  Efron-Stein inequalities and modified logarithmic Sobolev inequalities, 
we derive simple relations between the variance or the entropy of order statistics $X_{(k)}$  and moments of spacings $\Delta_k=X_{(k)}-X_{(k+1)}$.
When the sampling distribution has non-decreasing hazard rate (a condition that is satisfied by Gaussian, exponential, Gumbel, logistic distributions, etc, see \ref{dfn:hazard} for a definition), 
we are able to build on the connection between the fluctuations of  order statistics $X_{(k)}$ and spacings.
Combining Proposition \ref{thm:maud:1}  and R\'enyi's representation for order statistics, we connect the variance 
and the logarithmic moment generating function of $X_{(k)}$ with moments of spacings, Theorem \ref{prp:var:hazard:dec}
may be considered as 
an exponential Efron-Stein inequality for order statistics. 

In the framework of \textsf{EVT}, those relations are checked to be  asymptotically tight (see Section \ref{sec:assessment}).

In Section \ref{sec:gaussian-samples}, using
explicit bounds on the Gaussian  hazard rate, we derive Bernstein-like inequalities for 
the maximum and the median of a sample of independent Gaussian random variables
with a correct variance and scale factors (Proposition~\ref{prp:order-stat-gauss}). 
We provide non-asymptotic variance bounds for order statistics of Gaussian samples with the right order of magnitude in Propositions \ref{prp:var:gaussian},
and \ref{prp:cheap:tight}.

\section{Order statistics and spacings}
\label{sec:gener-purp-ineq}

Efron-Stein inequalities (\cite{efron:stein:1981}) allow us  to derive upper bounds on  the variance of 
functions of independent random variables.

\begin{thm}\textsf{(Efron-Stein inequalities.)}
Let $f \colon \mathbb{R}^n \rightarrow \mathbb{R}$ be measurable, and let $Z=f(X_1,\ldots,X_n)$.
Let  $Z_i= f_i (X_1,\ldots,X_{i-1},X_{i+1},\ldots, X_n)$ where $f_i \colon \mathbb{R}^{n-1} \rightarrow \mathbb{R}$ is an arbitrary measurable 
function.
Suppose $Z$ is square-integrable. Then 
\label{thm:ess:1}
\begin{displaymath}
\var[Z] \le \sum_{i=1}^n \esp \left[ \left( Z- Z_i \right)^2\right]  \pt 
\end{displaymath}  
\end{thm}
The quantity $\sum_{i=1}^n (Z-Z_i) ^2$ is called a jackknife estimate of variance. 

Efron-Stein inequalities form a special case of a more general collection of inequalities that encompasses 
the so-called modified logarithmic Sobolev inequalities \cite{ledoux:2001}. \\
Henceforth, the entropy of a non-negative random variable $X$ is defined by $\ent[X]= \esp [X \log X] - \esp X \log \esp X$.
 The next inequality from \cite{massart:2000} 
 has been used to derive a variety of 
concentration inequalities \cite{boucheron:lugosi:massart:2003}. 

  \begin{thm}\textsf{(Modified logarithmic Sobolev inequality.)}
    \label{thm:modlogsob}
    Let $\tau(x)=e^x-x-1$. Then for any $\lambda \in \mathbb{R}$, 
    \begin{displaymath}
    \ent\left[e^{\lambda Z}\right] = \lambda \esp \left[ Ze^{\lambda Z}\right] - \esp \left[ e^{\lambda Z}\right] \log \esp \left[ e^{\lambda Z} \right] \le \sum_{i=1}^n \esp \left[ e^{\lambda Z} \tau \left( -\lambda (Z-Z_i) \right) \right] \pt
    \end{displaymath}
  \end{thm}

Theorems \ref{thm:ess:1} and \ref{thm:modlogsob} provide a transparent connexion between 
 moments of order statistics and  moments of spacings. 

Henceforth, let $\psi\colon \mathbb{R}\rightarrow\mathbb{R}_+$ be defined by $\psi(x)=e^x\tau(-x)=1 + (x-1)e^x$.
  \begin{prp}\textsf{(Order statistics and spacings.)}
\label{thm:maud:1}
For all $1\le k \le n/2$, 
    \begin{displaymath}
      \var [X_{(k)}]  \leq k  \EXP \left[ (X_{(k)}-X_{(k+1)})^2\right] \, 
    \end{displaymath}
and   for all $\lambda \in \mathbb{R}$,
  \begin{displaymath}
  \ent \left[e^{\lambda X_{(k)} }\right] \le k \esp \left[e^{\lambda X_{(k+1)}} \psi( \lambda (X_{(k)} - X_{(k+1)})) \right]\pt
  \end{displaymath}
For all $n/2 < k \le n$, 
\begin{displaymath}
      \var [X_{(k)}]  \leq (n-k+1)  \EXP \left[ (X_{(k-1)}-X_{(k)})^2\right] \, 
    \end{displaymath}
 and for all $\lambda \in \mathbb{R}$,
  \begin{displaymath}
  \ent \left[e^{\lambda X_{(k)} }\right] \le (n-k+1) \esp \left[e^{\lambda X_{(k)}}  \tau( \lambda (X_{(k-1)} - X_{(k)}))  \right]\pt
  \end{displaymath}
\end{prp}

  \begin{proof}
 Let $Z=X_{(k)}$ and for $k\leq n/2$ define $Z_i$ as the rank $k$ statistic from subsample $X_1,\dots,X_{i-1},X_{i+1},\ldots, X_{n}$, that is 
$Z_i = X_{(k+1)}$ if $X_i\geq X_{(k)}$ and $Z_i=Z$ otherwise.
Apply  Theorem \ref{thm:ess:1}.

For $k > n/2$, define $Z_i$ as the rank $k-1$ statistic from $X_1, \ldots, X_{i-1}, X_{i+1},\ldots, X_n$, that is 
$Z_i = X_{(k-1)}$ if $X_i\leq X_{(k)}$ and $Z_i=Z$ otherwise.
Apply  Theorem \ref{thm:ess:1} again. 

For $k\leq n/2 $, define $Z$ and $Z_i$  as before,  apply  Theorem \ref{thm:modlogsob}: 
    \begin{eqnarray*}
    \ent \left[ e^{\lambda X_{(k)}}\right] &\le & k \esp \left[ e^{\lambda X_{(k)} }\tau \left( - \lambda ( X_{(k)} -X_{(k+1)} ) \right) \right] \\
    & = & k \esp \left[ e^{\lambda X_{(k+1)} }e^{\lambda (X_{(k)}-X_{(k+1)}) }\tau \left( - \lambda ( X_{(k)} -X_{(k+1)} ) \right) \right] \\
    &=& k \esp \left[e^{\lambda X_{(k+1)}} \psi( \lambda (X_{(k)} - X_{(k+1)})) \right] \pt
    \end{eqnarray*}
The proof of the last statement proceeds by the same argument.
  \end{proof}
Proposition \ref{thm:maud:1} can be fruitfully complemented by  R\'enyi's representation of order statistics (See \cite{dHaFei06} and references therein).

In the sequel, if $f$  is a monotone 
function from  $(a,b)$  (where  $a$ and $b$  may be infinite) 
to $(c,d)$, its generalized inverse  $f^\leftarrow : (c,d) \rightarrow (a,b)$ is defined by 
$f^\leftarrow(y) = \inf\{ x : a<x< b, f(x)\geq y \}\, $  (See \cite{dHaFei06} for properties of this transformation).  
\begin{dfn}
\label{dfn:U}
The $U$- transform of a distribution function $F$ is defined as a function on $(1,\infty)$ by $U=(1/(1-F))^{\leftarrow}$, 
\begin{math}
U(t) = \inf \{  x~:~F(x)\geq 1-1/t  \}= F^\leftarrow(1-1/t)\pt 
\end{math}
\end{dfn}

R\'enyi's representation   asserts that 
 the order statistics of a sample of independent exponentially distributed random variables are distributed 
as partials sums of independent exponentially distributed random variables. 
  \begin{thm}\textsf{(R\'enyi's representation)}
    \label{thm:renyi}
Let $X_{(1)}\geq\ldots\geq X_{(n)}$ be the order statistics of a sample from distribution $F$, let $U= (1/(1-{F}))^\leftarrow,$
 let $Y_{(1)} \geq Y_{(2)}\geq\ldots \geq Y_{(n)}$ be the order statistics of an independent sample of the exponential distribution, 
then 
\begin{displaymath}
\left( Y_{(n)}, \ldots, Y_{(i)},\ldots,
Y_{(1)} \right) \dis \big( \tfrac{E_n}{n},\ldots, \sum_{k=i}^n \tfrac{E_{k}}{k},\ldots, \sum_{k=1}^n \tfrac{E_{k}}{k}\big)
\end{displaymath}
where $E_1,\ldots,E_k$ are independent and identically distributed standard exponential random variables, and
\begin{displaymath}
      (X_{(n)}, \ldots, X_{(1)} ) \dis \big(U\circ \exp(Y_{(n)}), \ldots , U\circ \exp(Y_{(1)})  \big) \, . 
\end{displaymath}

  \end{thm}

We may readily test the tightness of propositions \ref{thm:maud:1}. 
By Theorem \ref{thm:renyi}, $Y_{(k)} = \frac{E_n}{n} +\ldots + \frac{E_k}{k}$ and 
\begin{math}
\var[Y_{(k)}] 
= \sum_{i=k}^n \tfrac{1}{i^2}\pt
\end{math} 
Hence, for any  sequence $(k_n)_n$ with $\lim_n k_n= \infty$, and $\limsup k_n/n< 1$,  
\begin{math}
\lim_{n\rightarrow \infty}k_n \var[Y_{(k_n)}]= 1 , 
\end{math}
while  by Proposition \ref{thm:maud:1}, 
\begin{math}
\var[Y_{(k)}] \le k \esp \big[ \big( {E_k}/{k} \big)^2 \big] = \tfrac{2}{k}.
\end{math}

The next  condition makes  combining  Propositions \ref{thm:maud:1} and Theorem \ref{thm:renyi} easy.
\begin{dfn}\textsf{(Hazard rate.)}
\label{dfn:hazard}  
The hazard rate of an absolutely continuous probability distribution with distribution function $F$ is:
\begin{math}
h = {f}/{\overline{F}} 
\end{math}
where  $f$ and $\overline{F}=1-F$ are respectively the density and the survival function associated with $F$. 
\end{dfn}

From elementary calculus, we get   $(U\circ \exp)'= 1/h(U\circ\exp)$ where $U=(1/(1-F))^\leftarrow$, which translates into
\begin{prp}
  \label{prp:incr:hazard}
Let  $F$ be an absolutely continuous distribution function with hazard rate $h$, 
let  $U=(1/(1-F))^\leftarrow$. 
Then,  $h$ is non-decreasing if and only if $U \circ \exp$ is concave.
\end{prp}

Observe that if the hazard rate  $h$ is non-decreasing, then 
for all $t>0$ and $x>0$, 
\begin{math}
U \left( \exp({t+x}) \right) - U \left( \exp(t) \right) \le {x}/{h(U(\exp(t)))}  \pt
\end{math}
Moreover, assuming that the hazard rate is non-decreasing  warrants negative association between spacings  and related order statistics.
\begin{prp}
\label{prp:hazard:rate:neg:assoc}
If $F$  has non-decreasing hazard rate, then the $k^{\text{th}}$ spacing $\Delta_k =X_{(k)}-X_{(k+1)}$ and $X_{(k+1)}$  are negatively associated:  
for any pair of non-decreasing functions $g_1$ and $g_2$,  
\begin{displaymath}
\esp[g_1(X_{(k+1)})g_2(\Delta_k)] \le \esp[g_1(X_{(k+1)})] \esp[g_2(\Delta_k)] \pt
\end{displaymath}
\end{prp}

\begin{proof}
Let $Y_{(n)},\ldots,Y_{(1)}$ be the order statistics of an exponential sample. Let $E_k=Y_{(k)}-Y_{(k+1)}$ be the $k^{\text{th}}$ spacing 
of the exponential sample. By Theorem \ref{thm:renyi},  $E_k$  and $Y_{(k+1)}$ are independent. 
Let $g_1$ and $g_2$ be two non-decreasing functions. By Theorem \ref{thm:renyi},
  \begin{eqnarray*}
\esp [g_1(X_{(k+1)}) g_2(\Delta_k) ] &= &\esp[ g_1 (U (e^{Y_{(k+1)}})) g_2 ( U(e^{E_k+Y_{(k+1)}}) - U(e^{Y_{(k+1)}}))] \\
&= &  \esp\left[ \esp\left[ g_1 (U (e^{Y_{(k+1)}})) g_2 ( U(e^{E_k+Y_{(k+1)}}) - U(e^{Y_{(k+1)}}))\mid Y_{(k+1)} 
\right] \right] \\
& = & \esp\left[ g_1 (U (e^{Y_{(k+1)}})) \esp\left[  g_2 ( U(e^{E_k+Y_{(k+1)}}) - U(e^{Y_{(k+1)}}))\mid Y_{(k+1)} 
\right] \right] \pt
  \end{eqnarray*}
The function $ g_1 \circ U \circ \exp$ is  non-decreasing. Almost surely, as the conditional distribution of $k E_k$ with respect 
to $Y_{(k+1)}$ is the  exponential distribution,  
\begin{eqnarray*}
  \esp\left[  g_2 ( U(e^{E_k+Y_{(k+1)}}) - U(e^{Y_{(k+1)}}))\mid Y_{(k+1)} 
\right] & = & \int_0^\infty e^{-x} g_2 (U(e^{\frac{x}{k}+Y_{(k+1)}}) - U(e^{Y_{(k+1)}})) \mathrm{d}x\pt
\end{eqnarray*}
As $F$  has non-decreasing hazard rate, $ U(\exp({x/k+y})) - U(\exp(y))=\int_0^{x/k} (U\circ\exp)'(y+z) \mathrm{d}z$ is non-increasing with respect to $y$. \\
This entails that $\esp\left[  g_2 ( U(e^{E_k+Y_{(k+1)}}) - U(e^{Y_{(k+1)}}))\mid Y_{(k+1)} 
\right]$ is a non-increasing function of $Y_{(k+1)}$. 
Hence, by Chebyshev's association  inequality,
 \begin{eqnarray*}
\lefteqn{\esp [g_1(X_{(k+1)}) g_2(\Delta_k) ] }
\\
& \leq &  \esp\left[ g_1 (U (e^{Y_{(k+1)}})) \right]\times
\esp \left[ \esp\left[g_2 ( U(e^{E_k+Y_{(k+1)}}) - U(e^{Y_{(k+1)}}))\mid 
Y_{(k+1)}\right] \right] \\
& = & \esp\left[ g_1 (U (e^{Y_{(k+1)}}))\right] \times \esp\left[g_2 ( U(e^{E_k+Y_{(k+1)}}) - U(e^{Y_{(k+1)}}))\right]\\
& = & \esp\left[ g_1 (X_{(k+1)})\right] \times \esp\left[g_2 (\Delta_k)\right] \pt 
  \end{eqnarray*}
\end{proof}

Negative association between order statistics and 
spacings allows us 
to establish our main result. 
 
\begin{thm}
\label{prp:var:hazard:dec}
Let    $X_1,\ldots,X_n$ be independently distributed according to $F$, let $X_{(1)}\geq \ldots\geq X_{(n)}$ be the order statistics and
let  $\Delta_k =X_{(k)}-X_{(k+1)}$ be the $k^{\text{th}}$ spacing.
Let $V_k= k \Delta_k^2$ denote the Efron-Stein estimate of the variance of $X_{(k)}$ (for $k=1,\ldots,n/2$). 

 If $F$  has non-decreasing hazard rate $h$, then for $1 \leq k\leq n/2$,
\begin{displaymath}
  \var\big[X_{(k)} \big] \leq \esp V_k \leq  \frac{2}{k}\EXP \left[  \big( \tfrac{1}{h(X_{(k+1)})}\big)^2\right] \, , 
\end{displaymath}
while for $k > n/2$,  
\begin{displaymath}
  \var\big[X_{(k)} \big] \leq \frac{2(n-k+1)}{(k-1)^2}\EXP \left[  \big( \tfrac{1}{h(X_{(k)})}\big)^2\right] \, .
\end{displaymath}
For $\lambda\geq 0$, and $1\leq k \leq n/2, $
\begin{equation}\label{eq:1}
 \log \esp e^{\lambda(X_{(k)} -\esp X_{(k)} ) } \leq \lambda \frac{k}{2} \esp \left[\Delta_k \left(e^{\lambda\Delta_k}-1\right) \right] 
= \lambda \frac{k}{2} \esp \left[\sqrt{\frac{V_k}{k}} \left(e^{\lambda\sqrt{V_k/k}}-1\right) \right] \, . 
\end{equation}
\end{thm}
Inequality~\eqref{eq:1}  may be considered as an exponential Efron-Stein inequality  for 
order-statistics:  it connects the logarithmic moment generating function of 
the $k^{\text{th}}$ order statistic with the exponential moments of the square root 
of the Efron-Stein estimate of variance $k\Delta_k^2$. This connection provides 
 correct bounds  for exponential distribution whereas the exponential Efron-Stein 
inequality described in \cite{boucheron:lugosi:massart:2003}
does not. This comes from the fact that negative association  between spacing 
and order statistics leads to an easy decoupling argument, there is no need to resort 
to the variational representation of entropy as in \cite{boucheron:lugosi:massart:2003}. It is then possible to 
carry out the so-called Herbst's argument in an effortless way. 
\begin{proof}Throughout the proof $Y_{(k)}, Y_{(k+1)}$ denote the $k^{\text{th}}$ and $k+1^{\text{th}}$ order 
statistics of an exponential sample of size $n.$ 
By Proposition~\ref{thm:maud:1}, using R\'enyi's representation (Theorem~\ref{thm:renyi}), and Proposition~\ref{prp:incr:hazard}, for 
$k\leq n/2$, 
\begin{eqnarray*}
  \var[ X_{(k)}] &\le& k \esp \left[ \left(U\left( e^{Y_{(k+1)}}e^{Y_{(k)}-Y_{(k+1)}} \right) - U \left( e^{Y_{(k+1)}} \right)\right)^2\right] \\
  & \le& k\esp \left[ \left( e^{Y_{(k+1)}} U' \left( e^{Y_{(k+1)}} \right) \right)^2 \left( Y_{(k)} -Y_{(k+1)} \right)^2\right]\\
& \leq & \frac{2}{k} \EXP \left[  \Big( \tfrac{1}{h(X_{(k+1)})}\Big)^2\right] \enspace ,
\end{eqnarray*}
as by Theorem \ref{thm:renyi}, 
\begin{math}
Y_{(k)}-Y_{(k+1)}  
\end{math}  is independent of $Y_{(k+1)}$ and exponentially distributed with scale parameter $1/k$.

By Proposition \ref{thm:maud:1} and~\ref{prp:incr:hazard}, as $\psi$ is non-decreasing over $\mathbb{R}_+$, 
\begin{eqnarray*}
  \ent \left[e^{\lambda X_{(k)} }\right] &\le & k \esp \left[e^{\lambda X_{(k+1)}} \psi( \lambda \Delta_k) \right]\\
& \leq & k  \esp \left[ e^{\lambda X_{(k+1)}} \right] \times
 \esp \left[ \psi( \lambda \Delta_k) \right] \\
& \leq & k  \esp \left[ e^{\lambda X_{(k)}} \right] \times
 \esp \left[ \psi( \lambda \Delta_k) \right] \, . 
\end{eqnarray*}
Multiplying both sides by $\exp(-\lambda \esp X_{(k)})$, 
\begin{eqnarray*}
  \ent \left[e^{\lambda (X_{(k)}-\esp X_{(k)}) }\right] & \leq & k  \esp \left[ e^{\lambda (X_{(k)}-\esp X_{(k)})} \right] \times
 \esp \left[ \psi( \lambda \Delta_k) \right] \pt
\end{eqnarray*}
Let $G(\lambda)= \esp e^{\lambda \Delta_k}$. Obviously, $G(0)=1$, and  
as $\Delta_k\geq 0$, $G$ and its derivatives are increasing on $[0,\infty)$, 
\begin{displaymath}
   \esp \left[ \psi( \lambda \Delta_k) \right]  =  1 - G(\lambda) + \lambda G'(\lambda)
 =  \int_0^\lambda sG^{\prime\prime}(s)\mathrm{d}s 
 \leq  G^{\prime\prime}(\lambda) \frac{\lambda^2}{2} \, . 
\end{displaymath}
 Hence, for $\lambda\geq 0,$
\begin{displaymath}
 \frac{\ent \left[e^{\lambda (X_{(k)}-\esp X_{(k)}) }\right] }{\lambda^2 \esp \left[ e^{\lambda (X_{(k)}-\esp X_{(k)})} \right]} =\frac{\mathrm{d}\frac{1}{\lambda} \log \esp e^{\lambda(X_{(k)} -\esp X_{(k)} ) } }{\mathrm{d}\lambda} 
\leq \frac{k}{2}\frac{\mathrm{d} G'}{\mathrm{d}\lambda} \, .
\end{displaymath}
Integrating both sides, using the fact that $\lim_{\lambda \rightarrow 0}  \frac{1}{\lambda} \log \esp e^{\lambda(X_{(k)} -\esp X_{(k)} ) } =0, $
\begin{displaymath}
  \frac{1}{\lambda} \log \esp e^{\lambda(X_{(k)} -\esp X_{(k)} ) }  \leq \frac{k}{2} (G'(\lambda)-G'(0))=  
\frac{k}{2} \esp \left[\Delta_k \left(e^{\lambda\Delta_k}-1\right) \right]  \, . 
\end{displaymath}
\end{proof}
\section{Asymptotic assessment}
\label{sec:assessment}

Assessing the quality of the variance  bounds from Proposition \ref{thm:maud:1} in full generality is not easy. However, Extreme Value Theory (\textsf{EVT})
describes 
a framework where the Efron-Stein estimates of variance are asymptotically of the right order of magnitude.  
\begin{dfn}
 The distribution function  $F$ belongs to a maximum domain of attraction  with tail  index $\gamma \in \mathbb{R}$
($F\in \textsf{MDA}(\gamma)$), 
if and only if there exists a non negative auxiliary function 
$a$ on $[1,\infty)$ such that for $x\in [0,\infty)$ (if $\gamma>0$), $\in [0,-1/\gamma)$ (if $\gamma<0$), $x\in \mathbb{R}$ (if $\gamma=0$)
$$\lim_n \mathbb{P} \left\{ \frac{\max(X_1,\ldots,X_n)-F^\leftarrow(1-1/n)}{a(n)} \leq x \right\}
= \exp(-(1+\gamma x)^{-1/\gamma}) \, . $$
If $\gamma=0,$ $(1+\gamma x)^{-1/\gamma}$ should be read as $\exp(-x).$
\end{dfn}
 If $F \in \mathsf{MDA}(\gamma)$ and has finite variance ($\gamma<1/2$),     the variance of 
$(\max(X_1,\ldots,X_n)-F^\leftarrow(1-1/n))/a(n)$ converges to the variance of the limiting extreme value distribution~\cite{dHaFei06}.

 Membership in a maximum domain of attraction
is characterized by the \emph{extended regular variation} property of $U= (1/(1-F))^{\leftarrow}$:  $F\in \mathsf{MDA}(\gamma)$ with auxiliary
function $a$ iff for all $x>0$
\begin{displaymath}
  \lim_{t\rightarrow \infty} \frac{U(tx)-U(t)}{a(t)}   = \frac{x^\gamma-1}{\gamma} \, ,
\end{displaymath}
where the right-hand-side should be read as $\log x$  when $\gamma=0$
%
\cite{dHaFei06}.

Using Theorem 2.1.1 and Theorem 5.3.1 from  \cite{dHaFei06}, and performing simple calculus,  we readily obtain
\begin{prp}\label{prop:assessment:var:bound}
  Assume $X_{(1)}\geq \ldots \geq  X_{(n)}$ are the order  statistics of an independent sample distributed according to $F$, where
 $F\in \textsf{MDA}(\gamma), \gamma < 1/2$ with auxiliary function~$a$.
Then 
\begin{displaymath}
  \lim_n \tfrac{\esp\left[\left( (X_{(1)} -X_{(2)})\right)^2\right]}{a(n)^2} = \tfrac{2\Gamma(2(1-\gamma))}{(1-\gamma)(1-2\gamma)}
\quad
\mathrm{ while }\quad 
\lim_n \tfrac{\var\left( X_{(1)} \right)}{a(n)^2} = \frac{1}{\gamma{^2}}  \left( \Gamma(1-2\gamma) -\Gamma(1-\gamma)^2\right)  . 
\end{displaymath}
For $\gamma=0$, the last expression  should be read as $\pi^2/6.$
\end{prp}
The asymptotic ratio between the Efron-Stein upper bound and the variance of 
$X_{(1)}$  converges toward a limit that depends only on $\gamma$ (for $\gamma=0$ this limit is $12/\pi^2 \approx 1.21$). 
When the tail index $\gamma<0$, the asymptotic ratio degrades as $\gamma \rightarrow -\infty,$ it scales like~$-4\gamma.$

\section{Order statistics of Gaussian samples}
\label{sec:gaussian-samples}
We now turn to the Gaussian setting. We will establish Bernstein inequalities for 
order statistics of absolute values of independent Gaussian random variables. 

A real-valued random variable $X$ is
said to be {\sl sub-gamma on the right tail with variance factor $v$ and
scale parameter $c$} if
\[
\log \esp e^{\lambda (X-\esp X)}\leq\frac{\lambda^2v}{2(1-c\lambda)  }
\text{ for every }\lambda\quad \mbox{such that} \quad 0<\lambda<1/c~.
\] Such a random variable satisfies a so-called Bernstein-inequality: for $t>0,$\\
\begin{math}
  \mathbb{P} \left\{ X\geq \esp X +\sqrt{2vt } + ct\right\}\leq \exp\left( -t\right)  . 
\end{math} A real-valued  random variable $X$ is
said to be {\sl sub-gamma on the left tail with variance factor $v$ and
scale parameter $c$}, if $-X$  is sub-gamma on the right tail with variance factor $v$ and
scale parameter $c$. A Gamma random variable with shape parameter $p$ and scale parameter $c$
(expectation $pc$ and variance $pc^2$) is sub-gamma on the right tail with variance factor $pc^2$ and scale factor $c$
while it is sub-gamma on the left-tail with variance factor $pc^2$ and scale factor $0$. The Gumbel distribution 
(with distribution function $\exp(-\exp(-x))$ is sub-gamma on the right-tail with variance factor $\pi^2/6$ and scale factor $1$,
it is sub-gamma on the left-tail with scale factor $0$ (note that this statement is not sharp, see Lemma \ref{lem:stupid} below). 

Order statistics of Gaussian samples provide an interesting playground for assessing Theorem~\ref{prp:var:hazard:dec}. 
Let $\Phi$ and $\phi$ denote respectively the standard Gaussian distribution  function and density. 
Throughout this section, let $\widetilde{U}\colon ]1,\infty) \rightarrow [0,\infty) $ be defined by 
\begin{math}
\widetilde{U}(t) = \Phi^\leftarrow(1-1/(2t)) \, , 
\end{math} $\widetilde{U}(t)$ is the $1-1/t$  quantile of the distribution  of the absolute value of 
a standard Gaussian random variable, or the $1-1/(2t)$ quantile of the Gaussian distribution.

\begin{prp} Absolute values of Gaussians have non-decreasing hazard rate : 
  \label{prp:absgaussian}\\
i) $\widetilde{U} \circ \exp$ is concave; \\
ii) For $y>0$, 
  \begin{math}
\phi(\widetilde{U}(\exp(y)))/\overline{\Phi}(\widetilde{U}(\exp(y)))   \ge {\sqrt{\kappa_1(y + \log 2)}} 
  \end{math}
 where $\kappa_1\geq 1/2.$\\
iii) For $t\geq 3,$
\begin{displaymath}
\sqrt{2 \log(2t)-\log\log(2t) -\log (4\pi) }\leq   \widetilde{U}(t) \leq\sqrt{2 \log(2t)-\log\log(2t) -\log \pi } \pt
\end{displaymath}
\end{prp}
\begin{proof}
i) As    $(\widetilde{U}\circ\exp)'(t)=\overline{\Phi}(\widetilde{U}(e^t))/\phi(\widetilde{U}(e^t)) $ it suffices to check that 
the standard Gaussian distribution has non-decreasing hazard rate  on $[0,\infty).$ Let $h=\phi/\overline{\Phi}$, 
by elementary calculus, for $x>0$, 
\begin{math}
h'(x)  = \left( \phi(x) - x\overline{\Phi}(x) \right) {\phi(x)}/{\overline{\Phi}^2(x)} \geq 0
\end{math}
where the last inequality is a well known fact. 

ii) For  $\kappa_1=1/2$, for $p\in (0,1/2]$, the fact that 
\begin{math}
 p\sqrt{\kappa_1\log 1/p}   \leq \phi\circ \Phi^{\leftarrow}(p) 
\end{math}
follows from $\phi(x) - x\overline{\Phi}(x)  \geq 0$ for $x>0$. Hence, 
\begin{displaymath}
 \frac{\overline{\Phi} ({\Phi}^\leftarrow(1-e^{-y}/2))}{\phi({\Phi}^\leftarrow(1-e^{-y}/2))}
= \frac{e^{-y}/2}{\phi ( \Phi^{\leftarrow}(e^{-y}/2))} 
 \le  \frac{1}{ \sqrt{\kappa_1(\log 2 + y)}} \pt
\end{displaymath}

iii) The first inequality can be deduced from $\phi\circ \Phi^{\leftarrow}(p) \leq p\sqrt{2\log 1/p}  \, , $
for $p\in (0,1/2)$ \cite{tillich2001dii}, the second from \begin{math}
 p\sqrt{\kappa_1\log 1/p}   \leq \phi\circ \Phi^{\leftarrow}(p) \pt 
\end{math} 
\end{proof}

The next proposition shows that when used in a proper way, Efron-Stein inequalities may provide 
seamless bounds on extreme, intermediate and central  order statistics of Gaussian samples.

\begin{prp}
\label{prp:var:gaussian}
Let $n\geq 3$, let $X_{(1)}\geq \ldots \geq X_{(n)}$ be the order statistics of absolute values of a  standard Gaussian sample, 
\begin{displaymath}
\text{For $1 \leq k\leq n/2$,}\quad \var [X_{(k)}] \le 
 \frac{1}{k\log 2}  \frac{8}{\log\tfrac{2n}{k} -\log (1+ \tfrac{4}{k} \log \log \tfrac{2n}{k})} 
\pt
\end{displaymath}
\end{prp}

By Theorem 5.3.1 from \cite{dHaFei06}, $\lim_n 2\log n \var[X_{(1)}]=\pi^2/6,$
while the above described upper bound on $\var[X_{(1)}]$ is equivalent to $(8/\log 2)/\log n $. 
If $\lim_n k_n =\infty$ while $\lim_n k_n/n=0$, Smirnov's lemma \cite{dHaFei06} implies that  
$\lim_n k(\widetilde{U}(n/k))^2 \var[X_{(k)}]=1.$ For the asymptotically normal 
median of absolute values,
 $\lim_n  (4 \phi(\widetilde{U}(2))^2n)\var [X_{(n/2)}]=1$ \cite{vandervaart:1998}. 
Again, the  bound in Proposition~\ref{prp:var:gaussian} has the correct order of magnitude.

\begin{lem}\label{lem:stupid}
  Let $Y_{(k)}$ be the $k^{\text{th}}$ order statistics of a sample of
  $n$ independent exponential random variables, let $\log 2<z< \log (n/k),$ then 
  \begin{displaymath}
    \mathbb{P}\left\{ Y_{(k+1)} \leq \log(n/k)-z \right\} \leq \exp\left( - \tfrac{k(e^z-1)}{4 }\right)\pt 
  \end{displaymath}
\end{lem}
\begin{proof}
  \begin{eqnarray*}
     \mathbb{P}\left\{ Y_{(k+1)} \leq \log(n/k)-z \right\} &= & \sum_{j=0}^k \binom{n}{j} \left(1-\frac{ke^{z}}{n}\right)^{n-j} \left(\frac{ke^{z}}{n} \right)^j \\
& \leq & \exp\left(-\frac{k (e^{z}-1)^2}{2  e^{z} }\right)  
  \end{eqnarray*}
since the right-hand-side of the first line is the probability that a binomial random variable with parameters $n$ and  $\frac{ke^{z}}{n}$ is less than $k$, 
which is  sub-gamma on the left-tail with variance factor less than $ke^z$ and scale factor $0$.
\end{proof}
\begin{proof}
By Propositions \ref{prp:var:hazard:dec} and \ref{prp:absgaussian}, letting $\kappa_1=1/2$
\begin{eqnarray*}
\var \left( X_{(k)} \right)& \le& \frac{2}{k}\esp \left[\frac{2}{ \log 2 + Y_{(k+1)}}\right] \\
&=& \frac{1}{\log2}\frac{4}{k} \mathbb{P} \Big \{ Y_{(k+1)} \le \log(n/k) -z \Big \}  + \frac{4}{k}  \frac{1}{\log \frac{n}{k}- z+  \log 2} \\
& \leq & \frac{4}{k \log 2 } \frac{1}{ \log\tfrac{2n}{k}} + \frac{4}{k}  \frac{1}{\log\tfrac{2n}{k} -\log (1+ \tfrac{4}{k} \log \log \tfrac{2n}{k})} 
\, , 
\end{eqnarray*}
where we used Lemma \ref{lem:stupid} with
 $z= \log (1+\frac{4}{k}\log \log \frac{2n}{k}).$
\end{proof}

Our next goal is to establish that the order statistics  of absolute values of independent Gaussian random variables
are sub-gamma on the right-tail with variance factor close to the Efron-Stein estimates of variance derived in Proposition~\ref{prp:absgaussian}
and scale factor not larger than the square root of the Efron-Stein  estimate of variance.

Before describing the consequences of Theorem \ref{prp:var:hazard:dec}, it is interesting to look at what 
can be obtained from R\'enyi's representation and exponential inequalities for sums of Gamma-distributed random variables. 
\begin{prp}\label{prp:cheap:tight}
  Let $X_{(1)}$ be the maximum of the absolute values of $n$ independent
  standard Gaussian  random variables, and let
  $\widetilde{U}(s)=\Phi^\leftarrow(1-1/(2s))$ for $s\geq 1.$ For $t>0$, 
  \begin{displaymath}
    \proba\left\{ X_{(1)} -\EXP X_{(1)} \geq t/(3\widetilde{U}(n)) +\sqrt{t}/\widetilde{U}(n) +\delta_n\right\} \leq \exp\left( -t\right) \, ,
  \end{displaymath}
where $\delta_n>0$ and 
  \begin{math}
    \lim_n (\widetilde{U}(n))^{3}\delta_n = \tfrac{\pi^2}{12}\pt 
  \end{math}
\end{prp}
This inequality looks  like what we are looking for:  ${\widetilde{U}(n)}(X_{(1)} -\EXP X_{(1)}) $ converges in distribution, but also in 
quadratic mean, or even according  to the  Orlicz norm defined by $x\mapsto \exp(|x|)-1$, 
toward a centered Gumbel distribution. The centered Gumbel distribution is sub-gamma on the right tail 
with variance factor $\pi^2/6$ and scale factor $1$,  we expect $X_{(1)}$ to satisfy  a Bernstein  inequality 
with variance factor of order $1/\widetilde{U}(n)^2$ and scale factor $1/\widetilde{U}(n)$. Up to the shift $\delta_n$, this is the content of the proposition. 
Note that the shift is asymptotically negligible with respect to the typical order of magnitude of the fluctuations.  The constants in the next proposition are not sharp enough  to make the next proposition competitive with Proposition~\ref{prp:cheap:tight}. Nevertheless it illustrates  that 
Proposition~\ref{prp:var:hazard:dec}  captures the correct  order of growth for the right-tail of Gaussian maxima. 

\begin{prp}
\label{prp:deviation:max:gaussian}
For  $n$ such that  the solution $v_n$ of equation 
$16/x+\log(1+2/x+4\log(4/x))=\log (2n)$ is smaller than 1, 
for all $0 \le \lambda <\frac{1}{\sqrt{v_n}}$, 
\begin{eqnarray*}
\log \esp e^{ \lambda(X_{(1)} - \esp X_{(1)})} \le \frac{v_n \lambda ^2 }{2(1-\sqrt{v_n}\lambda) } \pt
\end{eqnarray*}
For all $t >0$, 
\begin{displaymath}
\mathbb{P} \left \{ X_{(1)} - \esp X_{(1)} >  \sqrt{v_n} (t+\sqrt{2t}) \right \} \leq e^{-t}\pt
\end{displaymath}
\end{prp}

\begin{proof}
By Proposition \ref{prp:var:hazard:dec},
\begin{displaymath}
\log \esp e^{\lambda( X_{(1)} - \esp X_{(1)})} \le \frac{\lambda}{2} \esp \left[ \Delta \left( e^{\lambda \Delta}-1 \right) \right]
\end{displaymath}
where $\Delta = X_{(1)} - X_{(2)} \dis U(2e^{Y_{(2)} +E_1}) - U(2e^{Y_{(2)}})$, with $E_1$ is exponentially distributed and independent
of $Y_{(2)}$ which is distributed like the $2^{\text{nd}}$ largest order statistics of an exponential sample.\\
On the one hand, the conditional expectation 
\begin{displaymath}
\esp \left[ \left(U(2e^{E_1+ Y_{(2)}}) - U(2e^{Y_{(2)}})\right) \left(e^{\lambda (U(2e^{E_1+Y_{(2)}})-U(2e^{Y_{(2)}}))}-1\right) \vert Y_{(2)} \right]
\end{displaymath}
is a non-increasing function of $Y_{(2)}$. The maximum is achieved for $Y_{(2)}=0$, and it is equal to :
\begin{displaymath}
2 \int_{0}^{\infty} \frac{e^{-x^2/2}}{\sqrt{2 \pi}} x(e^{\lambda x}-1) \mathrm{d}x \le 2 \lambda e^{\frac{\lambda^2}{2}} \pt
\end{displaymath}
On the other hand, by Proposition \ref{prp:absgaussian}, 
\begin{displaymath}
U(2e^{E_1+Y_{(2)}}) - U(2e^{Y_{(2)}}) \le \frac{\sqrt{2} E_1}{ \sqrt{  (\log 2 + Y_{(2)})}} \pt
\end{displaymath}
Meanwhile,   for $0 \le \mu < 1/2$, 
\begin{displaymath}
\int_{0}^{\infty} \mu x (e^{\mu x}-1)e^{-x} \mathrm{d}x = \frac{\mu^2(2-\mu)}{(1-\mu)^2} \le \frac{2 \mu^2}{1-2\mu} \pt
\end{displaymath}
Hence,
\begin{displaymath}
 \begin{split}
   { \lambda \esp \left[\left(  U(2 e^{E_1+Y_{(2)}})- U(2e^{Y_{(2)}}) \right)
  \left(e^{\lambda(U(2 e^{E_1+Y_{(2)}})- U(2e^{Y_{(2)}}))}-1\right) \mid Y_{(2)}\right]}\qquad \qquad \qquad\\
  \leq 
 \frac{4\lambda^2}{\log 2+Y_{(2)}} \frac{1}{1-\tfrac{2\sqrt{2}\lambda}{\sqrt{\log 2+Y_{(2)}}}}\pt
 \end{split}
\end{displaymath}
 Letting $\tau= \log n - \log (1+2\lambda^2+4\log(4/v_n)),$
 \begin{eqnarray*}
 \log \esp \left[ e^{\lambda(X_{(1)}-\esp X_{(1)})} \right] &\le& \underbrace{\lambda^2 e^{\lambda^2/2} \mathbb{P} \left \{ Y_{(2)} \le \tau \right \}}_{:= \textsf{i}} +  \underbrace{\frac{4\lambda^2}{ \log 2+\tau} \frac{1}{1-\tfrac{2\sqrt{2}\lambda}{\sqrt{\log 2+\tau}}}}_{:= \textsf{ii}}\, .
 \end{eqnarray*}
 By Lemma \ref{lem:stupid}, 
\begin{math}
(\textsf{i}) \leq  \frac{v_n\lambda^2}{ 4} \pt
\end{math}\\
As $\lambda\leq 1/\sqrt{v_n} $,  $\log 2 +\tau\geq 16/v_n$ and 
\begin{math}
(\textsf{ii}) \le \frac{v_n \lambda^2}{4(1-\sqrt{v_n} \lambda)}\pt
\end{math}
\end{proof}

We may also use Theorem \ref{prp:var:hazard:dec} to provide a Bernstein inequality for the median of absolute values of a
Gaussian sample. We assume $n/2$ is an integer. 
\begin{prp}\label{prp:order-stat-gauss}
Let $v_n =8/(n \log 2)$. \\
For all $0 \le \lambda < n/(2 \sqrt{v_n})$, \\
\begin{displaymath}
\log \esp e^{\lambda(X_{(n/2)} - \esp X_{(n/2)})} \le \frac{v_n \lambda^2}{2(1- 2\lambda \sqrt{v_n/n})} \pt
\end{displaymath}
For all $t >0$, 
\begin{displaymath}
\mathbb{P} \left \{ X_{(n/2)} - \esp X_{(n/2)} > \sqrt{2 v_n t} + 2 \sqrt{v_n/n}t\right\} \le e^{-t} \pt
\end{displaymath}
\end{prp}

\begin{proof}
  By Proposition \ref{prp:var:hazard:dec},
  \begin{displaymath}
    \log \esp e^{\lambda(X_{(n/2)} - \esp X_{(n/2)})} \le \frac{n}{4} \lambda \esp \left[ \Delta_{n/2} \left( e^{\lambda \Delta_{n/2} } -1 \right) \right]
  \end{displaymath}
  where $\Delta_{n/2} = X_{(n/2)} - X_{(n/2+1)} \dis U \left(2e^{E/({n/2}) +
      Y_{(n/2 +1)}} \right) -U \left( e^{\lambda Y_{(n/2+1)}} \right)$ 
where $E$ is  exponentially distributed and independent of $Y_{(n/2+1)}$. 
  By Proposition \ref{prp:absgaussian},
  \begin{displaymath}
    \lambda \Delta_{n/2} \le \frac{\sqrt{2}\lambda E}{(n/2)\sqrt{\log 2 + Y_{(n/2 +1)}}} 
\le \frac{\sqrt{2}\lambda E}{(n/2) \sqrt{\log2}} =  \lambda\sqrt{\frac{v_n}{n}}E \pt
  \end{displaymath}
Reasoning as in the proof of Proposition
  \ref{prp:deviation:max:gaussian},
  \begin{displaymath}
    \log \esp e^{\lambda(X_{(n/2)} - \esp X_{(n/2)})} \le \frac{v_n \lambda^2}{2(1- {2\lambda\sqrt{v_n/n}})} \pt
  \end{displaymath}
\end{proof}

As the hazard rate $\phi(x)/\overline{\Phi}(x)$ of the Gaussian distribution tends to $0$ as $x$  tends to $-\infty$, the preceding approach 
does not work when dealing with order statistics of Gaussian samples.
Nevertheless,  Proposition \ref{prp:var:gaussian} paves the way to simple bounds on the variance of maxima of Gaussian samples. 
\begin{prp}
  \label{prp:max:gaussian}
  Let $X_1,\ldots,X_{n}$ be $n$ independent and identically distributed standard Gaussian random variables, let $X_{(1)} \ge \ldots \ge X_{(n)}$ be the order statistics.\\
For all $n \geq 11$, 
\begin{eqnarray*}
\var[X_{(1)}] &\le&  \frac{8/\log 2}{\log(n/2)-\log (1+4\log\log(n/2))} + 2^{-n}+ \exp(-\tfrac{n}{8})+ \frac{4\pi}{n} \pt 
\end{eqnarray*}
\end{prp}

\begin{proof}[Proof of Proposition~\ref{prp:max:gaussian}.]
We may generate $n$ independent standard Gaussian random variables in two steps:  first generate $n$ independent
random signs ($\epsilon_1,\ldots, \epsilon_n$:  $\proba\{ \epsilon_i=1\}=1-\proba\{ \epsilon_i=-1\}=1/2$), then generate 
absolute values ($V_1,\ldots,V_n$), the resulting sample ($X_1, \ldots, X_n$) 
is obtained as $X_i=\epsilon_i V_i$.  Let $N$ be the number of positive random signs. 
\begin{displaymath}
  \var(X_{(k)}) =  \underbrace{\esp \left[ \var\big( X_{(k)} \mid \sigma(N) \big)\right]}_{\textsf{i}} +
 \underbrace{\var\left( \esp\left[ X_{(k)}\mid \sigma({N})\right] \right)}_{\textsf{ii}}\, . 
\end{displaymath}
Conditionally on $N=m$, if $k\leq m$, $ X_{(k)}$ is distributed as the $k^{\text{th}}$ order statistic of a sample 
of $m$ independent absolute values of Gaussian random variables. If $k>m$, $X_{(k)}$ is negative, its 
conditional variance is equal to the variance of the statistics of order $n-k+1$ in a sample of size $n-m$. 
Hence, letting $V_{(k)}^m$ denote the $k^{\text{th}}$ order statistic of a sample 
of $n$ independent absolute values of Gaussian random variables.

For $k=1$, $\textsf{(ii)}\ll\textsf{(i)}$, as  
\begin{eqnarray*}
  \textsf{(i)} & = & \sum_{m=1}^n \binom{n}{m} 2^{-n} \var\big( V^m_{(1)} \big) +\binom{n}{0} 2^{-n} \var\big( V^{n}_{(n)} \big) \\
&\leq & \sum_{m=1}^{n/4} \binom{n}{m} 2^{-n} + \var\big( V^{n/4}_{(1)} \big) + 2^{-n} \\
& \leq & \exp(-n/8) +\var\big( V^{n/4}_{(1)} \big) + 2^{-n}  \, 
\end{eqnarray*}
where the middle term can be upper bounded using  Proposition \ref{prp:var:gaussian}.
 
Let $H_N$ be  the random harmonic number $H_N= \sum_{i=1}^N1/i,$
\begin{eqnarray*}
  \textsf{(ii)}  & = &  \esp_{N,N'}  \left[\Big( \esp[X_{(1)} \mid N]- \esp[X_{(1)} \mid N']\Big)_+^2 \right]
\\
&= & \esp_{N,N'}  \left[ \left(\esp \left[ \widetilde{U}\Big( \exp\Big( \textstyle{\sum_{i=1}^N \tfrac{E_i}{i}}\Big) \Big)-\widetilde{U}\Big( \exp\Big( \textstyle{\sum_{i=1}^{N'} \tfrac{E_i}{i}}\Big) \Big)\right] \right)_+^2\right]\\ 
& \leq & \esp_{N,N'} \left[ \left( 1/h\big(\widetilde{U}\big( \exp\big( \textstyle{\sum_{i=1}^{N'} \tfrac{E_i}{i}}\big) \big)\big)   (H_N-H_{N'})\right)^2_+\right]\\
& \leq & \frac{\pi}{2} \esp_{N,N'} \Big[ (H_N-H_{N'})_+^2\Big] \\
& = & \frac{\pi}{2} \var(H_N) \, . 
\end{eqnarray*}
Now, as $N= \sum_{i=1}^n (1+\epsilon_i)/2$,  
letting $Z=H_N$ and $Z_i= \sum_{j=1}^{N-(1+\epsilon_i)/2} 1/j$, by Efron-Stein inequality, $\var{Z} \leq \esp[0\wedge 1/N].$
Finally, using Hoeffding inequality in a crude way leads to  $\esp[0\wedge 1/N]\leq \exp(-n/8)+ 4/n\leq 8/n$. We may conclude 
by $\textsf{ii}\leq (4\pi)/n$. 
\end{proof}

\begin{rem}
Trading simplicity for tightness, sharper  bounds on $\textsf{ii}$  could be derived and show that $\textsf{ii}=O(1/(n\log n))$.    
\end{rem}

\bibliographystyle{abbrv}


\appendix 

\label{sec:extreme-value-theory}

\section{Proof of Proposition \ref{prp:cheap:tight}}
\label{sec:proof-prop-}

Let $Y_{(1)}$ denote the maximum of a sample of absolute values of 
independent exponentially distributed random variables, so that $X_{(1)}\sim \widetilde{U}(e^{Y_{(1)}}).$

Thanks to the concavity of $\widetilde{U}\circ \exp$, and the
  fact that $\overline{\Phi}(x)/\phi(x)\leq 1/x$ for $x>0,$
  $\widetilde{U}(\exp(Y_{(1)}))-\widetilde{U}(\exp(H_n))\leq (Y_{(1)}-H_n)/(\widetilde{U}(\exp(H_n)).$ Now, $Y_{(1)}$ satisfies a Bernstein inequality
with variance factor not larger than  $\var(Y_{(1)})\leq \pi^2/6\leq 2 $ and scale factor not larger than $1$, so 
  \begin{displaymath}
\proba\{
  Y_{(1)} -H_n \geq t \widetilde{U}(\exp(H_n))\}\leq
  \exp\Big(-\frac{t^2\widetilde{U}(\exp(H_n))^2}{2(2+ t\widetilde{U}(\exp(H_n))/3)}\Big).
\end{displaymath}
Agreeing on $\delta_n= \widetilde{U}(e^{H_n}) -\EXP \widetilde{U}(e^{Y_{(1)}})$, using the monotonicity of $x^2/(2+x/3)$ and $\log n \leq H_n$, we obtain the first part of the proposition.

The second-order extended regular variation condition satisfied  the Gaussian distribution (See \cite{dHaFei06} Exercise 2.9 p. 61) allows to assert
\begin{math}
\widetilde{U}(n)^3\left(   \widetilde{U}(e^{x}n)-\widetilde{U}(n)-\tfrac{x}{\widetilde{U}(n)}\right) \rightarrow - \frac{x^2}{2}-x \, . 
\end{math}
This suggests  that
\begin{displaymath}
\widetilde{U}(\exp(H_n))^3\left(   \EXP \widetilde{U}(e^{Y_{(1)}-H_n}e^{H_n})-\widetilde{U}(e^{H_n})-\frac{\EXP Y_{(1)}-H_n}{\widetilde{U}(\exp(H_n))}\right) \rightarrow - \frac{\pi^2}{12} \, ,
\end{displaymath}
or that the order of magnitude of $\widetilde{U}\big(e^{H_n}\big)-\EXP X_{(1)}$ is $O(1/\widetilde{U}(n)^3)$ which is small with respect to $1/\widetilde{U}(n)$. 

Use Theorem B.3.10 from \cite{dHaFei06}, to ensure that for all $\delta, \epsilon>0$, for $n,x$ such that  $\min(n, n\exp(x))\geq t_0(\epsilon,\delta)$, 
\begin{math}
\left| \widetilde{U}(n)^3\left(   \widetilde{U}(e^{x}n)-\widetilde{U}(n)-\tfrac{x}{\widetilde{U}(n)}\right)  + \frac{x^2}{2}+x \right| \leq \epsilon \exp(\delta |x|) \, . 
\end{math} 
The probability that $Y_{(1)}\leq H_n/2$  is less than $\exp(-H_n/3)$.
When for $Y_{(1)}-H_n\leq 0$,  the  supremum  of 
  \begin{displaymath}
\left|  \widetilde{U}(e^{H_n})^3\left(   \widetilde{U}(e^{Y_{(1)}-H_n}e^{H_n})-\widetilde{U}(e^{H_n})-\frac{ Y_{(1)}-H_n}{\widetilde{U}(e^{H_n})}\right) +
\frac{ (Y_{(1)}-H_n)^2}{2}+{ Y_{(1)}-H_n}{} \right|\, ,
\end{displaymath}
is achieved when $Y_{(1)}-H_n=-H_n$, it is less than  $4 (\log n) ^2$. 
The dominated convergence theorem allows to conclude.

\end{document}